\documentclass[]{interact}

\usepackage{epstopdf}
\usepackage[caption=false]{subfig}

\usepackage[numbers,sort&compress]{natbib}
\bibpunct[, ]{[}{]}{,}{n}{,}{,}
\renewcommand\bibfont{\fontsize{10}{12}\selectfont}
\makeatletter
\def\NAT@def@citea{\def\@citea{\NAT@separator}}
\makeatother

\theoremstyle{plain}
\newtheorem{theorem}{Theorem}[section]
\newtheorem{lemma}[theorem]{Lemma}
\newtheorem{corollary}[theorem]{Corollary}
\newtheorem{proposition}[theorem]{Proposition}

\theoremstyle{definition}
\newtheorem{definition}[theorem]{Definition}
\newtheorem{example}[theorem]{Example}

\theoremstyle{remark}

\DeclareMathOperator{\prox}{prox}
\DeclareMathOperator{\Fix}{Fix}
\begin{document}


\title{Degenerate preconditioned backward-backward splitting for inclusion problem}

\author{
\name{Pankaj Gautam\textsuperscript{a}\thanks{CONTACT Pankaj Gautam. Email: pgautam908@gmail.com} and V. Vetrivel\textsuperscript{b}}
\affil{\textsuperscript{a}Department of Mathematical Sciences, NTNU Trondheim, Norway; \\
\textsuperscript{b}Department of Mathematics, IIT Madras, India}
}

\maketitle

\begin{abstract}
In this work, we introduce the notion of warped Yosida regularization and study the asymptotic behaviour of the orbit of dynamical systems generated by warped Yosida regularization, which includes Douglas-Rachford dynamical system. We analyze an algorithm where the inclusion problem is first approximated by a regularized one and then the preconditioned regularization parameter is reduced to converge to a solution of original problem. We propose and investigate backward-backward splitting using degenerate preconditioning for monotone inclusion problems. The applications provide a tool for finding a minima of a preconditioned regularization of the sum of two convex functions. 
\end{abstract}

\begin{keywords}
Monotone operator; Yosida regularization; Preconditioning; Backward-backward splitting.
\end{keywords}

\section{Introduction}

In this work, we aim to solve the inclusion problem$\colon$
\begin{align}\label{1.1}
    \text{find}~(x,y)\in \mathcal{H}^2~\text{such that}~(0,0)\in (I- R+\lambda(M^{-1}A\times M^{-1}B))(x,y),
\end{align}
where $A,B:\mathcal{H}\to 2^\mathcal{H}$ are two set-valued operators, $R:\mathcal{H}^2\to \mathcal{H}^2$ is an operator defined by $(x,y)\mapsto (y,x)$, and $M:\mathcal{H}\to \mathcal{H}$ is linear, bounded, self-adjoint and positive-semidefinite operator, and $\mathcal{H}$ is a real Hilbert space. The dual inclusion problem of (\ref{1.1}) is 
\begin{align}\label{e1.2}
    \text{find}~(x^*,y^*)\in \mathcal{H}^2~\text{such that}~(0,0)\in ((I-R)^{-1}+(M^{-1}A\times M^{-1}B)^{-1}\circ (I/\lambda))(x^*,y^*).
\end{align}
The solution of problems (\ref{1.1}) and (\ref{e1.2}) are related to warped Yosida regularization of structured monotone inclusion problem of the form:
\begin{align}\label{1.3}
    \text{find}~x\in \mathcal{H}~\text{such that}~0\in (A+B)x. 
\end{align}
To find the solution of problem (\ref{1.3}), the Yosida regularization of problem (\ref{1.3}) is firstly introduced by Mahey and Tao \cite{mahey1993partial}, and Yosida regularization of an operator is introduced in \cite{moreau1965proximite, yosida2012functional}.\\
If $M=I$, (\ref{1.1}) becomes the problem of finding 
\begin{align}\label{1.4}
    (x,y)\in \mathcal{H}^2~\text{such that}~(0,0)\in I-R+\lambda (A\times B)(x,y),
\end{align} proposed in \cite{bauschke2005asymptotic}, in which the authors have shown the asymptotic behaviour of the sequences generated by composition of resolvents in the connection with the solution of (\ref{1.4}) and its dual by using Yosida regularization of problem (\ref{1.3}). The resolvent and Yosida regularization of a maximal monotone operator are very important tools in nonlinear analysis. Bui et al. \cite{bui2020warped} have proposed a generalized resolvent called warped resolvent by using an auxiliary operator. The main motivation to define warped resolvent is to construct and investigate splitting methods by different choices of kernels. Bredies et al. \cite{bredies2022degenerate} have analyzed Chambolle-Pock, Forward Douglas-Rachford and Peaceman-Rachford splitting methods with the help of warped resolvent by using linear kernel. 

In section \ref{sc2}, we introduce and investigate the properties of warped Yosida approximation.  In section \ref{sc3}, we analyze the existence, uniqueness, weak convergence of the dynamical systems governed by the warped Yosida regularization and observed Douglas-Rachford dynamical system as a particular case. Section \ref{sc4} describes the preconditioned Yosida regularization of the monotone inclusion problem and analyzes the preconditioned backward-backward splitting for problem (\ref{1.1}). Section \ref{sc5} provides backward-backward splitting methods for an optimization problem. 
\section{Preminilaries}
This section is devoted to some important definitions and results from nonlinear analysis and operator theory.
Throughout the paper, $\mathcal{G}(T)$ is used to denote the graph of the set-valued operator $T:\mathcal{H}\to 2^{\mathcal{H}} $, which is defined as $\mathcal{G}(T)$ =\{$(x,Tx):x\in D(T)$\}, where $D(T)$ denotes the domain of $T$. Let $Z(T)$ and $\Fix(T)$ denote the zero and set of fixed point of an operator $T$, respectively. Symbol $\mathbb{N}$ and $\mathbb{R}$ are used to denote the set of natural numbers and set of real numbers, respectively.
\begin{definition}
	Let $T: \mathcal{H} \to 2^\mathcal{H}$ be a set-valued operator. Then $T$ is said to be monotone if, $\forall$ $ x_1,x_2 \in \mathcal{H}$ and $u_1\in T(x_1), u_2 \in T(x_2)$,
	\begin{align*}
	0 \leq \langle x_1-x_2,u_1-u_2 \rangle.
	\end{align*}
	$T$ is said to be maximally monotone if there exists no monotone operator $S:\mathcal{H} \to 2^\mathcal{H}$ such that $\mathcal{G}(S)$    properly contains $\mathcal{G}(T)$, i.e., for every $(x_1,u_1)\in \mathcal{H} \times \mathcal{H}$,
	\begin{align*}
	(x_1,u_1) \in \mathcal{G}(T) \Leftrightarrow  \langle x_1-x_2, u_1-u_2 \rangle \geq 0, \   \forall(x_2,u_2) \in \mathcal{G}(T).
	\end{align*}
\end{definition}

Let $M$ be a bounded linear operator on $\mathcal{H}$. $M$ is said to be self-adjoint if $M^*= M$, where $M^*$ denotes the transpose conjugate of operator $M$.
A self-adjoint operator $M$ on $\mathcal{H}$ is said to be positive definite if $\langle M(x), x\rangle > 0$  for every nonzero $x \in \mathcal{H}$ (\cite{limaye1996functional}).
Define the $M$-inner product $\langle \cdot,\cdot \rangle_M $ on $\mathcal{H}$ by
$\langle x,y\rangle_M = \langle x,M(y)\rangle \text{ for all } x,y \in \mathcal{H}.$ The corresponding $M$-norm  is defined by $\|x\|^2_M=\langle x, Mx \rangle$ for all $x \in \mathcal{H}$.

\begin{definition} Let $D$ be a nonempty subset of  $H$, $T:D\rightarrow \mathcal{H}$ an operator and  $M:\mathcal{H} \to \mathcal{H}$   a positive definite operator. Then
	$T$ is said to be
	\begin{enumerate}
		\item [(i)] nonexpansive with respect to $M$-norm if
		$$\|Tx_1-Tx_2\|_M \leq \|x_1-x_2\|_M  ~ \forall x_1,x_2 \in \mathcal{H};$$
		\item [(ii)]  $M$-cocoercive    if
		\begin{equation*}\label{P2As1}
		\|Tx_1-Tx_2\|^2_{M^{-1}} \leq \langle x_1-x_2, Tx_1-Tx_2 \rangle,  \text{   for all $x_1, x_2 \in \mathcal{H}$}.
		\end{equation*}
	\end{enumerate}
\end{definition}
\begin{lemma}\label{2l1}\cite{bot2016second}
	Let $C$ be a nonempty subset of $\mathcal{H}$ and $x:[0,\infty)\to \mathcal{H}$ be a map. Assume that
	\begin{itemize}
		\item[(i)] $\lim\limits_{t\to\infty}\|x(t)-x^*\|$ exists, for every $x^*\in C$;
		\item[(ii)] every weak sequential cluster point of the map $x$ is in $C$.
	\end{itemize}
	Then there exists $x_\infty\in C$ such that $x(t)\rightharpoonup x_\infty$ as $t\to\infty$.
\end{lemma}
Throughout the paper we consider operator $M:\mathcal{H}\to \mathcal{H}$ as a preconditioner, i.e., a linear, bounded, self-adjoint and positive semidefinite operator, and assume that $\mathcal{H}$ is a real Hilbert space with norm $\|\cdot\|$ and inner product $\langle\cdot, \cdot\rangle$.
\begin{definition}\cite{bredies2022degenerate} An admissible preconditioner for the operator $T:\mathcal{H} \to 2^\mathcal{H}$ is a linear, bounded, self-adjoint and positive semi-definite operator $M: \mathcal{H}\to \mathcal{H}$ such that warped resolvent 
\begin{align}\label{eq1.3}
    J_{\gamma T}^M = (M+\gamma T)^{-1}\circ M
\end{align}
is single-valued and has full domain. 
\end{definition}
It is easy to check that $J_{\gamma T}^M$ is neither everywhere defined $(J_{\gamma T}^M 0= \emptyset)$ nor single-valued. For this reason, instead of imposing the maximal monotonicity of $T$, we directly require (\ref{eq1.3}), which, in the context of splitting methods, is a reasonable assumption.
\begin{proposition}\cite{bredies2015preconditioned}\label{pro2.1}
Let $M:\mathcal{H}\to \mathcal{H}$ be an admissible preconditioning and $T:\mathcal{H}\to 2^{\mathcal{H}}$ be an operator such that $M^{-1}A$ is $M$-monotone. Then $J_{\gamma T}^M$ is $M$-firmly nonexpansive, i.e.,
\begin{align*}
    \|J_{\gamma T}^Mx-J_{\gamma T}^My\|_M^2+\|(I-J_{\gamma T}^M)x-(I-J_{\gamma T}^M)y\|_M^2\le \|x-y\|_M^2, ~\forall x,y\in \mathcal{H}.
\end{align*}
\end{proposition}
\section{Warped Yosida regularization}\label{sc2}
In this section, we define warped Yosida regularization and provide its properties and some characterizations.
\begin{definition}
	Let $X$ be a reflexive Banach space with dual space $X^*$. Assume that $(\emptyset \neq)C\subseteq X$, $M:C\to {X^*}$ and $T: X\to 2^{X^*}$ are such that $\operatorname{ran}(M)\subset \operatorname{ran} (T+\gamma M)$ and $T+\gamma M$ is injective. For any $\gamma\in (0,\infty)$, warped Yosida regularization of $T$ with kernel $M$ is defined by $T_{\gamma}^M= \frac{1}{\gamma}\left( M-M\circ J_{\gamma T}^M\right)$, where $J_{\gamma T}^M$ is the  warped resolvent of $T$ \cite{bui2020warped}.
\end{definition}
\begin{example}
	Let $C$ be a non-empty subset of $X$ and $\phi: X\to (-\infty,\infty ]$ be a proper convex lower semicontinuous map. Let $\gamma >0$. Assume that $M:C\to X^*$ is an operator with $\operatorname{ran}(M)\subset \operatorname{ran}(M+\gamma \partial \phi)$ and $M+\gamma\partial \phi$ is injective. Then warped Yosida regularization of $\partial \phi$ is $(\partial \phi)_\gamma^M= \frac{1}{\gamma}(M-M\circ \prox_{\gamma\partial\phi }^M)$, where $\prox_{\gamma\partial\phi }^M= (M+\gamma \partial \phi)^{-1}\circ M$.
	\newline
	Let $M$ be an injective operator, then warped Yosida regularization of $\partial \phi$ is described by the following variational inequality:
	\begin{align*}
		z= (\partial \phi)_{\gamma}^M\Leftrightarrow (\forall y\in X)~\langle y-x+\gamma M^{-1} z, \gamma z\rangle +\phi(x-\gamma M^{-1}z)\le \phi(z)~\forall (x,z)\in X\times X.
	\end{align*}
\end{example}
\begin{example}
	Let $T:X\to 2^{X^*}$ be a maximal monotone operator such that $Z(T)\neq \emptyset$. Suppose that $f:X\to (-\infty, \infty]$ is an admissible function such that $D(T)\subset int~D (f)$. Set $M= \nabla f$. Then $T_{\gamma}^{\nabla f}$ is a well defined warped Yosida regularization defined in \cite{reich2010two}. 
\end{example}
Now, we provide an example of warped Yosida regularization with respect to different choices of the admissible function $f$. 
\begin{example}
	Let $A:(0,\infty)\to \mathbb{R}$ be a monotone mapping. Define an admissible function (Boltzmann-Shannon entropy) $\mathcal{BS}:(0,\infty)\to (0,\infty)$ as $x\mapsto x \log x-x$. Then warped resolvent of $A$ is \cite{sabach2012iterative}
	$$ J_{A}^{\mathcal{BS}}x= (\log + A)^{-1}\circ Ax= xe^{Ax}$$
	and the warped Yosida regularization is
	\begin{align*}
		A^{\mathcal{BS}}x&= \nabla {\mathcal{BS}}x -\nabla \mathcal{BS}\circ J_{A}^{\mathcal{BS}}x\\
		&= \log x- \log (xe^{Ax})\\
		&= \log x- \log x- \log e^{Ax}\\
		&=-Ax.
	\end{align*}
\end{example} 

Now, we explore some properties and characteristics of warped Yosida regularization.
\begin{proposition}\label{prop1.1}
	Let $T:\mathcal{H}\to 2^{\mathcal{H}}$ and $M:\mathcal{H}\to \mathcal{H}$ be an admissible preconditioner. Then we have the following:
	\begin{itemize}
		\item[(i)] $\left( J_{\gamma T}^M(x), T_{\gamma}^M(x) \right)\in \mathcal{G}(T) $, $x\in \mathcal{H}$.
		\item [(ii)] $0\in T(x)$ if and only if $ 0\in T_{\gamma}^M(x)$, $x\in \mathcal{H}$.
        \item [(iii)] $T_{\gamma}^M= \left(\gamma M^{-1}+T^{-1} \right)^{-1}$.
		\item [(iv)] $T_{\gamma}^M= J_{\gamma^{-1}T^{-1}}^{M^{-1}}\circ \gamma^{-1} M$. 
		\item [(v)] $T_{\gamma+\lambda}^M= \left( T_{\gamma}^M \right)_{\lambda}^M$.
		\item [(vi)] Let $x,y\in\mathcal{H}$. Then $x= T_{\gamma}^My\Leftrightarrow ( y-\gamma M^{-1}x,x)\in \mathcal{G}(T)$.
        \item [(vii)] $\overline{D(T)}$ is convex, $ D(J_{\gamma T}^M)= D(T_\gamma^M)$ and $\operatorname{ran}(J_{\gamma T}^M)= D(T)$.
	\end{itemize}
\end{proposition}
\begin{proof}
	\begin{itemize}
		\item[(i)] For $x\in \mathcal{H}$, we have
		\begin{align*}
			&J_{\gamma T}^M(x)= (M+\gamma T)^{-1}\circ M(x)
			\Leftrightarrow M(x)\in (M+\gamma T)\circ J_{\gamma T}^M(x)\\[2pt]
			\Leftrightarrow& \frac{1}{\gamma}\left( M-M\circ J_{\gamma T}^M\right)(x)\in T\left( J_{\gamma T}^M(x)\right)
			\Leftrightarrow T_{\gamma}^M(x)\in T\left( J_{\gamma T}^M(x)\right).
		\end{align*}
		\item[(ii)] For $x\in \mathcal{H}$,
		\begin{align*}
			&0\in T(x)\Leftrightarrow 0\in \gamma T(x)
			\Leftrightarrow M(x)\in (M+\gamma T)(x)\\
			\Leftrightarrow& x\in (M+\gamma T)^{-1}\circ M(x)
			\Leftrightarrow M(x)\in M(J_{\gamma T}^M(x))\\
			\Leftrightarrow& 0\in \left( M-M\circ J_{\gamma T}^M\right)(x)
			\Leftrightarrow 0\in \gamma T_{\gamma}^Mx\Leftrightarrow 0\in T_{\gamma}^Mx.
		\end{align*}
       \item[(iii)] Let $x,y\in \mathcal{H}$. Indeed,
		\begin{align}
			&x\in T_{\gamma}^M y \Leftrightarrow x\in \frac{1}{\gamma}\left( M-M\circ J_{\gamma T}^M\right)(y)\nonumber\\
			\Leftrightarrow &\gamma x\in \left( M-M\circ J_{\gamma T}^M\right)(y)\nonumber
			\Leftrightarrow M(y)-\gamma x\in (M\circ J_{\gamma T}^M)(y)\nonumber\\
			\Leftrightarrow & M^{-1}(M(y)-\gamma x)\in J_{\gamma T}^M(y)\nonumber
			\Leftrightarrow y-\gamma M^{-1}(x) \in (M+\gamma T)^{-1}\circ M(y)\nonumber\\
			\Leftrightarrow& (M+\gamma T)(y-\gamma M^{-1}(x))\in M(y)\nonumber
			\Leftrightarrow M(y)+ \gamma T(y)-\gamma x-\gamma^2 T\circ M^{-1}(x)\in M(y)\nonumber\\
			\Leftrightarrow& x\in T(y-\gamma M^{-1}(x))\label{5.01}
			\Leftrightarrow x\in (\gamma M^{-1}+T^{-1})^{-1}(y).
		\end{align}
  	   \item[(iv)] Let $x,y\in \mathcal{H}$. From part (i),
		\begin{align*}
			&x\in T_{\gamma}^My \Leftrightarrow  x\in (\gamma M^{-1}+T^{-1})^{-1}(y)\\
			\Leftrightarrow & x\in ( M^{-1}+\gamma^{-1}T^{-1})^{-1}\circ M^{-1}(\gamma^{-1}M(y))\\
			\Leftrightarrow & x\in J_{\gamma^{-1}T^{-1}}^{M^{-1}}\circ \gamma^{-1} M(y).
		\end{align*}
		\item[(v)] From part(i), for $x,y\in \mathcal{H}$,
		\begin{align*}
			x\in T_{(\gamma+\lambda)}^M(y)\Leftrightarrow & x\in ((\gamma+\lambda) M^{-1}+T^{-1})^{-1}(y)\\
			\Leftrightarrow& x\in T(y-(\gamma+\lambda) M^{-1}(x))\\
			\Leftrightarrow& x\in T(y-\gamma M^{-1}y-\lambda M^{-1}y)\\
			\Leftrightarrow& x\in T_{\gamma}^M(y-\lambda M^{-1}(y))\\
			\Leftrightarrow& x\in \left( T_{\gamma}^M \right)_{\lambda}^M y.
		\end{align*}
		\item [(vi)] For $x,y\in \mathcal{H}$ and from (\ref{5.01}), we have
		\begin{align}\label{e1.5}
			y\in T_{\gamma}^M x \Leftrightarrow y\in T(x-\gamma M^{-1}y)\Leftrightarrow ( x-\gamma M^{-1}y,y)\in \mathcal{G}(T).
		\end{align}
	\end{itemize}
\end{proof}

\begin{proposition}\label{pro2.3}
	Let $T:\mathcal{H}\to 2^{\mathcal{H}}$ be a maximal monotone operator and $M:\mathcal{H}\to \mathcal{H}$ be an admissible preconditioner and $\gamma>0$. Then, we have the following:
	\begin{itemize}
        \item [(i)] $T_\gamma^M$ is $\gamma$-cocoercive with respect to $M$,
        i.e.,
        \begin{align}
            \gamma\|T_{\gamma}^Mx_1-T_\gamma^Mx_2\|^2_{M^{-1}}\le \langle x_1-x_2, T_\gamma^Mx_1-T_\gamma^Mx_2\rangle.
        \end{align}and hence maximal monotone and $L$-Lipschitz continuous, for some $L>0$. 
        \item [(ii)] $J_{\gamma T}^M$ is $M$-nonexpansive.
        \item [(iii)] Let $x_\gamma=J_{\gamma T}^Mx$. Then $x_{\gamma}\to P^M_{\overline{dom(T)}}x$ as $\lambda \downarrow 0$, where $P^M_{\overline{dom(T)}}x= \arg\min_{y\in \overline{dom(T)}}\|x-y\|_M$.
		\item [(iv)] For $x,y,~and~p\in \mathcal{H}$
		\begin{eqnarray}{
				(y,p)= \left( J_{\gamma T}^Mx, T_{\gamma}^Mx\right) \Leftrightarrow\left\{
				\begin{array}{lc@{}c@{}r}
					(y, p)\in \mathcal{G}(T),\nonumber\\[2pt]
					x= y+\gamma M^{-1}p.
				\end{array}\right.
		}\end{eqnarray} 
	\end{itemize}
	\end{proposition}
\begin{proof}
	\begin{itemize}
       \item [(i)] 
       Let $x_1,x_2\in \mathcal{H}$.
       Then 
       \begin{align*}
           \|T_{\gamma}^Mx_1-T_\gamma^Mx_2\|^2_{M^{-1}}&= \langle M^{-1}(T_{\gamma}^Mx_1-T_\gamma^Mx_2), T_{\gamma}^Mx_1-T_\gamma^Mx_2\rangle\\
           &=\frac{1}{\gamma^2} \langle (I-J_{\gamma T}^M)x_1-(I-J_{\gamma T}^M)x_2, (I-J_{\gamma T}^M)x_1-(I-J_{\gamma T}^M)x_2\rangle_M\\
           &= \frac{1}{\gamma^2} \|(I-J_{\gamma T}^M)x_1-(I-J_{\gamma T}^M)x_2\|^2_M.
       \end{align*}
       Since $I-J_{\gamma T}^M$ is $M$-firmly nonexpansive, hence
       \begin{align*}
           \|T_{\gamma}^Mx_1-T_\gamma^Mx_2\|^2_{M{-1}}&\le \frac{1}{\gamma^2} \langle x_1-x_2, (I-J_{\gamma T}^M)x_1- (I-J_{\gamma T}^M)x_2\rangle_M\\
           &= \frac{1}{\gamma^2}\langle x_1-x_2, M(I-J_{\gamma T}^M)x_1- M(I-J_{\gamma T}^M)x_2\rangle\\
           &= \frac{1}{\gamma}\langle x_1-x_2, T_\gamma^M x_1-T_\gamma^Mx_2\rangle.
       \end{align*}
       In the similar manner, we can also show that 
       \begin{align}\label{e1.7}
           \langle x_1-x_2, T_\gamma^M x_1-  T_\gamma^M x_2\rangle_M\ge \gamma\|T_{\gamma}^M x_1- T_\gamma^M x_2\|^2.  
       \end{align}
       \item [(ii)] It follows from \cite[Proposition 3.1]{dixit2021accelerated}
       \item [(iii)] Let $\gamma\in (0,1)$ and $(y,v)\in \mathcal{G}(T)$. Then by the monotonicity of $T$, we have 
       \begin{align*}
          &\langle x_{\gamma}-y, Mx-Mx_{\gamma}-\gamma v\rangle \ge 0 \\
          \Rightarrow & \|x_\gamma-y\|^2_M\le \langle x_{\gamma}-y, x-y\rangle_{M}+\|x_{\gamma}-y\|_{M}\|M^{-1}v\|_{M}\\
          \Rightarrow & \|x_{\gamma}-y\|_M\le \|x-y\|_{M}+\gamma \|M^{-1}v\|_M,
       \end{align*}
which yields $(x_{\gamma})_{\gamma\in (0,1)}$ is bounded. Rest part of the proof follows from the proof of \cite[Theorem 23.48]{bauschke2017correction}.

		\item [(iv)] From the definitions of warped resolvent and warped Yosida regularization, we have
		\begin{eqnarray}{
				\left\{
				\begin{array}{lc@{}c@{}r}
					y= J_{\gamma T}^Mx,\nonumber\\[2pt]
					p= T_{\gamma}^Mx,
				\end{array}\right.\Leftrightarrow
				\left\{
				\begin{array}{lc@{}c@{}r}
					(y, \gamma^{-1}(Mx-My)\in \mathcal{G}(T),\nonumber\\[2pt]
					p= \frac{1}{\gamma}(Mx-My),
				\end{array}\right.\Leftrightarrow
				\left\{
				\begin{array}{lc@{}c@{}r}
					(y, p)\in \mathcal{G}(T),\nonumber\\[2pt]
					x= y+\gamma M^{-1}p.
				\end{array}\right.
		}\end{eqnarray}
   
		\end{itemize}
\end{proof}
\begin{proposition}
	Let $T:\mathcal{H}\to 2^{\mathcal{H}}$ be a maximal monotone operator and $M:\mathcal{H}\to \mathcal{H}$ be an adimissable preconditioner. Then, for any $x\in \mathcal{H}$, we have the following:
	\begin{itemize}
		\item [(i)] $J_{\mu T}^Mx\left( \frac{\mu}{\gamma}x+(1-\frac{\mu}{\gamma})J_{\gamma T}^Mx\right) = J_{\gamma T}^Mx$.
		\item [(ii)]  $\|J_{\gamma T}^Mx-J_{\mu T}^Mx\|\le \frac{\beta}{\alpha}\left( 1-\frac{\mu}{\gamma}\right)\|M^{-1}T_{\gamma}^Mx\|.$
	\end{itemize}
\end{proposition}
\begin{proof}
	\begin{itemize}
		\item[(i)] 
		For $x\in \mathcal{H}$ and $\mu=\lambda \gamma$, we have
		\begin{align*}
			&x\in ((M+\gamma T)^{-1}\circ M)^{-1}\circ (M+\gamma T)^{-1}\circ Mx\\
			&\Leftrightarrow x\in M^{-1}\circ (M+\gamma T)\circ J_{\gamma T}^Mx\\
			& \Leftrightarrow (Mx- M\circ J_{\gamma T}^Mx)\in \gamma T(J_{\gamma T}^Mx)\\
			& \Leftrightarrow \lambda Mx- (1-\lambda)M\circ J_{\gamma T}^Mx\in (M+\mu T)J_{\gamma T}^Mx\\
			&\Leftrightarrow (M+\mu T)^{-1}\circ M(\lambda x+ (1-\lambda)J_{\gamma T}^Mx )=J_{\gamma T}^Mx\\
			& \Leftrightarrow J_{\mu T}^M (\lambda x+ (1-\lambda)J_{\gamma T}^M x)=  J_{\gamma T}^Mx.
		\end{align*}
		Putting $\lambda=\frac{\mu}{\gamma}$, we have
		\begin{align*}
			J_{\mu T}^Mx\left( \frac{\mu}{\gamma}x+(1-\frac{\mu}{\gamma})J_{\gamma T}^Mx\right) = J_{\gamma T}^Mx.
		\end{align*}
		\item [(ii)] From (i), we get
			\begin{align*}
			\|J_{\gamma T}^Mx-J_{\mu T}^Mx\|&= \|J_{\mu T}^Mx\left( \frac{\mu}{\gamma}x+(1-\frac{\mu}{\gamma})J_{\gamma T}^Mx\right) - J_{\mu T}^Mx \|\\
			&\le \frac{\beta}{\alpha}\left\|\left( 1-\frac{\mu}{\gamma}\right)(x-J_{\gamma T}^Mx)  \right\| \\
			&= \frac{\beta}{\alpha} \left( 1-\frac{\mu}{\gamma}\right)\|M^{-1}T_{\gamma}^Mx\|. 
		\end{align*}
	\end{itemize}
\end{proof}

\section{Motivation to define warped Yosida regularization}\label{sc3}
Consider the monotone inclusion problem$\colon$
\begin{align}\label{1.5}
    \text{find}~u\in \mathcal{H}~\text{such that}~0\in Tu,
\end{align}
where $T:\mathcal{H}\to 2^{\mathcal{H}}$ is a maximal monotone operator defined on Hilbert space $\mathcal{H}$. The differential inclusion to solve problem (\ref{1.5}) is
\begin{align}\label{1.6}
    \begin{cases}
    \dot{u}(t)\in -Tu(t)\\
    u(t_0)=u_0\in \mathcal{H}.
    \end{cases}
    \end{align}
    In general, the differential inclusion \eqref{1.6} is not well-posed. For example, if we consider an operator $T:\mathbb{R}^2\to \mathbb{R}^2$ defined by
    \begin{align*}
        T(x,y)=(-y, x), 
    \end{align*}then the orbit of (\ref{1.6}) does not converge to zero of $T$. But we have $Z(T)= Z(T_{\gamma}^M)$. In order to overcome this difficulty, we shall consider the following dynamical system 
  \begin{align}\label{1.7}
    \begin{cases}
        \dot{u}(t) +T_{\gamma}^Mu(t)=0\\
        u(t_0)=u_0\in \mathcal{H}.
        \end{cases}
    \end{align}
    In the next result, we study the properties of the orbit $u(t)$ generated by the dynamical system (\ref{1.7}). 
    \begin{proposition}
    Let $T:\mathcal{H}\to 2^{\mathcal{H}}$ be a maximal monotone operator with $T^{-1}(0)\neq \emptyset$ and $M$ be an admissible preconditioner. Let $\gamma>0$ and $u_0\in D(T)$. Then we have the following identities for the unique solution $u(t)$ of dynamical systems (\ref{1.7})$\colon$
    \begin{itemize}
        \item [(i)] $u\in C^{1}(\mathbb{R}, \mathcal{H})$ and $u(t)\in D(T)$, $\forall t>0$.
       \item [(ii)] The orbit $u$ is bounded and $\dot{u}\in L^{2}([t_0,\infty);\mathcal{H})$.
         \item [(iii)] $u(t)$ converges weakly to $u^*$, for some $u^*\in T^{-1}(0)$.
    \end{itemize}
    \end{proposition}
    \begin{proof}
        \begin{itemize}
            \item [(i)] It follows from the fact that $T_{\gamma}^M$ is Lipschitz continuous. 
            \item [(ii)] Let $u^*\in T^{-1}(0)$. Define an anchor 
            \begin{align*}
                h(t)=\frac{1}{2}\|u(t)-u^*\|^2_M. 
            \end{align*}
            From (\ref{1.7}), we get 
            \begin{align*}
                \dot{h}(t)+\langle u(t)-u^*, T_\gamma^M x(t)\rangle_M=0.
            \end{align*}
            Since $u^*\in T^{-1}(0)= (T_\gamma^M)^{-1}(0)$ and by (\ref{e1.7}), we deduce
            \begin{align}
                &\dot{h}(t)+\gamma\|T_\gamma^M u(t)\|^2\le 0\nonumber\\
                \Rightarrow ~& \dot{h}(t)+\gamma\|\dot{u}(t)\|^2\le 0.\label{2.11}
            \end{align}Hence $t\mapsto h(t)$ is a monotonically decreasing function. Since $t\mapsto h(t)$ is a locally absolutely continuous function, there exists $N_1\in \mathbb{R}$ such that 
            \begin{align*}
                h(t)\le N_1 ~\text{for all}~t\in [t_0, \infty),
            \end{align*}which concludes that $h(t)$ is a bounded function and hence $u(t)$ is also bounded function.
            Integrating (\ref{2.11}), we get a real number $N_2$ such that
            \begin{align}
                h(t)+\gamma\int_{t_0}^{t}\|\dot{u}(t)\|^2\le N_2~\text{for all}~t\in [t_0, \infty).
            \end{align}Using the boundedness of $h$, we conclude that $\dot{u}\in L^2([t_0,\infty); \mathcal{H})$.
            
            \item [(iii)] By Lemma \ref{2l1}, and maximal monotonicity of $T_{\gamma}^M$, the orbit weakly converges to $u^*$ for some $u^*\in T^{-1}(0)$.
           
        \end{itemize}
    \end{proof}

    Interestingly, from $T_{\gamma}^M$ one can obtain different splitting based dynamical systems by different choices of operators $T$ and $M$, e.g., Douglas-Rachford, forward-Douglas-Rachford dynamical systems. \\
    Let $u=(x,y)\in \mathcal{H}:=H^2$, where $H$ is a real Hilbert space. Consider $T:H^2 \to 2^{H^2}$ and $M:H^2 \to H^2$ be the operators defined by \begin{align}\label{1.8}
    T=\begin{bmatrix}
    \alpha A & I\\
    -I & (\alpha B)^{-1}
    \end{bmatrix},~~~
    M=\begin{bmatrix}
    I & -I\\
    -I & I
    \end{bmatrix},
    \end{align}
    where $A,B:H\to 2^H$ are maximal monotone operators. With this choice of operator $T$, problem (\ref{1.5}) is converted into 
    \begin{align}
        \text{find}~u\in H~\text{such that}~0\in (A+B)u,
    \end{align}
    equivalently, $0\in (A+B)x$ if and only if there exists $y\in H$ such that $0\in \alpha Ax+y$ and $0\in \alpha Bx$.
    Now, for $T$ and $M$ defined by (\ref{1.8}) we have
    \begin{align*}
        &J_{T}^M\\
        &= \begin{bmatrix}
            \alpha A+I& 0\\
            -2I & (\alpha B)^{-1}+I
        \end{bmatrix}^{-1}\circ \begin{bmatrix}
            I& -I\\
            -I & I
        \end{bmatrix}\\[2pt]
        &=\begin{bmatrix}
            (\alpha A+ I)^{-1} & -(\alpha A+I)^{-1}\\
            2((\alpha B)^{-1}+I)^{-1}(\alpha A+I)^{-1}- ((\alpha B)^{-1}+I)^{-1} & -2((\alpha B)^{-1}+I)^{-1}(\alpha A+I)^{-1}+ ((\alpha B)^{-1}+I)^{-1}
        \end{bmatrix}\\[2pt]
        &= \begin{bmatrix}
            J_{\alpha A}& -J_{\alpha A}\\
            2J_{(\alpha B)^{-1}}J_{\alpha A}- J_{(\alpha B)^{-1}} & -2J_{(\alpha B)^{-1}}J_{\alpha A}+ J_{(\alpha B)^{-1}}
        \end{bmatrix}.
    \end{align*}
    So, \begin{align*}
        T^M=&\begin{bmatrix}
            I& -I\\
            -I& I
        \end{bmatrix} -\begin{bmatrix}
            J_{\alpha A}- 2J_{(\alpha B)^{-1}}J_{\alpha A}+ J_{(\alpha B)^{-1}} & - J_{\alpha A} +2J_{(\alpha B)^{-1}}J_{\alpha A}- J_{(\alpha B)^{-1}}\\
            -J_{\alpha A}+2J_{(\alpha B)^{-1}}J_{\alpha A}- J_{(\alpha B)^{-1}} & J_{\alpha A}-2J_{(\alpha B)^{-1}}J_{\alpha A}+ J_{(\alpha B)^{-1}}
        \end{bmatrix}\\[2pt]
         =& \begin{bmatrix}
            \frac{I-R_{\alpha A}R_{\alpha B}}{2} & \frac{R_{\alpha A}R_{\alpha B}-I}{2}\\
            \frac{R_{\alpha A}R_{\alpha B}-I}{2} &  \frac{I-R_{\alpha A}R_{\alpha B}}{2}
        \end{bmatrix},
    \end{align*}
    where $R_{\alpha A}$ and $R_{\alpha B}$ are reflected resolvents of the operators $A$ and $B$, respectively. 
Hence from dynamical system (\ref{1.7}), we get
\begin{align*}
\begin{cases}
        \dot{x}(t)=  -\frac{I-R_{\alpha A}R_{\alpha B}}{2} (x(t)-y(t))\\
         \dot{y}(t)= -\frac{I-R_{\alpha A}R_{\alpha B}}{2} (y(t)-x(t))\\
         x(0)=x_0, y(0)=y_0.
    \end{cases}
\end{align*}
By substituting $z(t)= x(t)-y(t)$, we obtain that  
\begin{align*}
    \begin{cases}
        \dot{z}(t)+z(t)= R_{\alpha A} R_{\alpha B}z(t)\\
        z(0)=z_0,
    \end{cases}
\end{align*}
   which is a Douglas-Rachford (without preconditioing) dynamical system investigated in \cite{gautamparameterized}. 
   
\section{Preconditioned backward-backward splitting}\label{sc4}
    This section generalizes the regularization of sum of two monotone operators \cite{mahey1993partial} using preconditioner. 
A preconditioned regularization of monotone inclusion problem 
\begin{align}\label{eq1.1}
    \text{find}~x\in \mathcal{H}~\text{such that}~ 0\in Ax+Bx 
\end{align}
is 
\begin{align}\label{eq1.2}
    \text{find}~x\in \mathcal{H}~\text{such that}~ 0\in A_\lambda^M x+Bx.
\end{align}
Note that,  
$\text{if}~ 0\in A_\lambda^M x+Bx$,
    \begin{align}
    M^{-1}\circ (M+\lambda B)x&\in J^M_{\lambda A}x\nonumber\\
    x&\in (M+\lambda B)^{-1}\circ M\circ J_{\lambda A}^Mx\nonumber\\
    x&\in J_{\lambda B}^M \circ J_{\lambda A}^M x. 
\end{align}
We study the convergence of sequence $\{x_n\}$ defined by$\colon$ $\{x_n\}\in \mathcal{H}$ and $\{x_n\}$ solves the problem (\ref{eq1.2}). Assume that problem (\ref{eq1.2}) has a solution for $\lambda>0$. Then for each $n$, $\{x_n\}$ satisfies:
\begin{align}\label{eq1.7}
    x_n= J_{\lambda B}^M \circ J_{\lambda A}^M x_n.
\end{align}
Let $y_n= A^M_\lambda x_n$. In the next result, we study the convergence analysis of sequences $\{x_n\}$ and $\{y_n\}$. 
\begin{theorem}
Let $A,B:\mathcal{H}\to 2^\mathcal{H}$ be two maximal monotone operators. Then we have the following:
\begin{itemize}
    \item [(i)]Assume that $x$ is a solution of (\ref{eq1.1}) and $y\in Ax\cap (-Bx)$. Then $\|y_n\|\le \|y\|$ for any $\lambda>0$.
    \item [(ii)] Let $x$ be a limit point of $\{x_n\}$ and the sequence $\{y_n\}$ be bounded. Then, $x$ solves \eqref{eq1.1}.
    \item [(iii)] If $\{x_n\}$ has a limit point, then problem \eqref{eq1.1} has a solution if and only if $\{y_n\}$ is bounded. 
    \item [(iv)] If $\{x_n\}$ is bounded and problem \eqref{eq1.1}  has a unique solution, then $x_n\to x^*$ and $y_n\to y^*$ which is the element of minimum norm in $Ax^*\cap (-Bx^*)$.
\end{itemize}

\end{theorem}
\begin{proof}
\begin{itemize}
    \item [(i) ] As $y\in -Bx$ and $y_n\in -B x_n$, so using the monotonicity of $B$, we have
\begin{align}\label{eq1.8}
    \langle y_n-y, x_n-x\rangle \le 0.
\end{align}
Also 
\begin{align*}
    y_n= A_\lambda^M x_n= \frac{1}{\lambda} \left( M x_n- M\circ J_{\lambda A}^M x_n\right)
\end{align*}
which implies that 
\begin{align}\label{1.15}
    M x_n= \lambda y_n + M\circ J_{\lambda A}^M x_n
\end{align}
From \eqref{eq1.8} and \eqref{eq1.9}, we have
\begin{align*}
    \langle y_n-y, x_n-x\rangle &= \langle y_n-y, \lambda M^{-1}y_n+ J_{\gamma A}^M x_n-x \rangle \le 0,
\end{align*}
i.e.,
\begin{align*}
    \langle y_n-y, \lambda M^{-1} y_n\rangle &\le -\langle y_n-y, J_{\lambda A}^M x_n- x\rangle.
\end{align*}
Since $y\in Ax$, $y_n= A_\lambda^M x_n\in A(J_{\lambda A}^M)$ and $A$ is monotone, we get 
\begin{align*}
    \|y_n\|^2\le \|y\|\|y_n\|. 
\end{align*}
\item [(ii)] By Proposition \ref{prop1.1}(vii) and (\ref{eq1.7}), we have $x_n\in D(B)$. Again, by Proposition \ref{prop1.1}(vii) and (\ref{1.15}), we obtain $x_n \in \lambda M^{-1} y_n+ D(A)$. As operator $M$ and sequence $\{y_n\}$ are bounded, the first term of the sequence $\{x_n\}$ tends to $0$ as $\lambda \downarrow 0$ and limit point $x$ must belong to $\overline{D(A)}$. Hence $x\in \overline{D(A)}\cap \overline{D(B)}$. 

Since $\{y_n\}$ is bounded, we can consider a convergent subsequence $\{y_{n_k}\}$ of $\{y_n\}$. Let $y$ be the limit of $\{y_{n_k}\}$. 
Consider $z_{n_k}= J_{\lambda A}^Mx_{n_k}$. Then $y_{n_k}\in A z_{n_k}$ for any $k$. Now,
\begin{align*}
    \|z_{n_k}-x\|_M&=\|z_{n_k}- J_{\lambda A}^M x+ J_{\lambda A}^M x-x\|_M\\
    &\le \|J_{\lambda A}^M x_{n_k}- J_{\lambda A}^M x\|_M+\|J_{\lambda A}^M x-x\|_M\\
    & \le \|x_{n_k}-x\|_M+\|J_{\lambda A}^M x-x\|_M.
\end{align*}
Since $\|\cdot\|$ and $\|\cdot\|_M$ are equivalent norms, by Proposition \ref{pro2.3}(iii), the second term of the above inequality tends to $0$ as $\lambda\downarrow 0$ (as $x\in \overline{D(A)}$), and hence $z_{n_k}\to x$. 

As $A$ is a closed map, $y_{n_k}\to y$ and $z_{n_k}\to x$, hence $y\in Ax$. Also $-y_{n_k}\in Bx_{n_k}$ and $B$ is a closed map, $-y\in Bx$. Thus $x$ solves (\ref{eq1.1}). 
\item [(iii)] It follows from 
(i) and (ii).
\item [(iv)] Let $x^*$ is an unique solution of the problem (\ref{eq1.1}). Since $\{x_n\}$ is bounded and from part (ii), we conclude that its limit point is a solution of (\ref{eq1.1}). Hence $x^*$ is the unique limit point of the sequence $\{x_n\}$ and $x_n\to x^*$. Suppose that $y^*$ is an element of minimum norm in $Ax^*\cap (-Bx^*)$. Then from part (i), we get
\begin{align*}
    \|y_n\|\le \|y^*\|.
\end{align*}
Let $y_n\to y$. Then $\|y\|\le \|y^*\|$. Also from part (2), we have $y\in Ax^*\cap(-Bx^*)$, and hence $y=y^*$. hence $y_n\to y^*$.
\end{itemize}
\end{proof}
In the similar manner, we can show that the problem:
\begin{align}\label{1.13}
    \text{find}~y\in \mathcal{H}~\text{such that}~0\in A y+ B_{\lambda}^M y
\end{align}
is also a preconditioned regularization of problem (\ref{eq1.1}). The dual of the problems (\ref{eq1.2}) and (\ref{1.13}) are 
\begin{align}
    \text{find} ~x^*\in H~\text{such that}~ 0\in \widetilde{A_\lambda^M}x^*+ B^{-1}x^*,
\end{align}
and 
\begin{align}
    \text{find}~y^*\in H ~\text{such that}~ 0\in A^{-1}y^*+ \widetilde{B_\lambda^M}y^*,
\end{align}respectively, where $\widetilde{A}= (-I)\circ A^{-1}\circ(-I)$.  

In the next result, we show the relation between the solution set of the problems (\ref{1.1}), (\ref{e1.2}), (\ref{eq1.2}) and (\ref{1.13}), which are denoted by $S$, $S^*$, $E$ and $F$, respectively. 
\begin{proposition}\label{pro4.1}
    The following identities hold:
    \begin{itemize}
    \item [(i)]$E=\Fix (J_{\lambda B}^M J_{\gamma A}^M)=J_{\lambda B}^M (F)$, and $F=\Fix (J_{\lambda A}^M J_{\lambda B}^M)=J_{\lambda A}^M (E)$.
    \item [(ii)]$S= (F\times E) \cap \mathcal{G}(J_{\lambda B}^M)$
    \item [(iii)] $S^*= \{(\lambda u^*, \lambda v^*)\}$ such that  $v^*=-u^*$, where $u^*= J^{M^{-1}}_{\frac{A^{-1}+\widetilde{B}}{\lambda}}(0)$ and $v^*=J^{M^{-1}}_{\frac{\widetilde{A}+B^{-1}}{\lambda}}(0)$. 
    \item [(iv)] $S^*= -R\circ S$.
    \item [(v)] $J_{\lambda B}^M|_F :F\to E: x\mapsto x+\lambda M^{-1}u^*$ is a bijective map with inverse $J_{\lambda A}^M|_E :E\to F: y\mapsto y+\lambda M^{-1}v^*$.
    \end{itemize}
\end{proposition}
\begin{proof}\begin{itemize}
    \item [(i)] Let $x\in F$. Then 
    \begin{align*}
        0\in Ax+B_{\lambda}^M x \Leftrightarrow &  0\in \lambda Ax+ Mx- M\circ J_{\lambda B}^Mx\\
        \Leftrightarrow & M^{-1}\circ (M+\lambda A)x\in J_{\lambda B}^Mx \\
        \Leftrightarrow & x\in \Fix (J_{\lambda A}^M J_{\lambda B}^M).
    \end{align*}
    In the similar manner, one can show that $E= \Fix(J_{\lambda B}^M J_{\lambda A}^M)$. For $x\in \Fix(J_{\lambda B}^M J_{\lambda A}^M)$, we have $J_{\lambda A}^Mx\in \Fix(J_{\lambda A}^M J_{\lambda B}^M)$. Hence $J_{\lambda A}^M(E)\subset F$. Similarly, $J_{\lambda B}^M(F)\subset E$. Again by these last two inclusions, we obtain that $F=J_{\lambda A}^M(E)$ and $E=J_{\lambda B}^M(F)$. 
    \item [(ii)] Let $(x,y)\in S$. Then 
    \begin{align*}
        &-R(x,y)\in \lambda(A\times B)(x,y)\\
        & \Leftrightarrow (My, Mx)\in (Mx, My)+\lambda(A\times B)(x,y)\\
        & \Leftrightarrow My\in Mx+\lambda Ax ~ \text{and}~ Mx= My+\lambda By\\
        & \Leftrightarrow x= J_{\lambda A}^My~\text{and}~ y= J_{\lambda B}^M x \\
        & \Leftrightarrow x\in \Fix (J_{\lambda A}^M J_{\lambda B}^M), ~ y\in \Fix (J_{\lambda B}^M J_{\lambda A}^M)~\text{and}~y=J_{\lambda B}^M x\\
        & \Leftrightarrow (x,y)\in F\times E~\text{and}~(x,y )= \mathcal{G}(J_{\lambda B}^M).
    \end{align*}
    \item[(iii)] For $(x^*,y^*)\in S^*$, we have 
    \begin{align*}
        (0,0)\in R^{-1}(x^*,y^*)+\left(A^{-1}\left(\frac{x^*}{\lambda}\right)\times B^{-1}\left(\frac{x^*}{\lambda}\right)\right).
    \end{align*}
    Then, there exists $(x,y)\in S$ such that 
    \begin{align*}
        \begin{cases}
            (x,y)\in A^{-1}\left(\frac{x^*}{\lambda}\right)\times B^{-1}\left(\frac{x^*}{\lambda}\right),\\
            (-x,-y)\in R^{-1}(x^*, y^*),
        \end{cases} \Rightarrow 
        \begin{cases}
            (x^*,y^*)\in \lambda (A\times B)(x,y),\\
            (-x^*, -y^*)= R (x,y),
        \end{cases}
    \end{align*}
    which implies that
    \begin{align}\label{2.21}
        \begin{cases}
            x\in A^{-1}(x^*/\lambda),\\
            -y\in \widetilde{B}(x^*/\lambda),\\
            x^*= M(y-x)= -y^*.
        \end{cases}
    \end{align}
    On the other hand, 
    \begin{align}\label{2.22}
        \begin{cases}
            -x\in \widetilde{A}(y^*/\lambda),\\
            y\in B^{-1}(y^*/\lambda),\\
            x^*= M(y-x)= -y^*.
        \end{cases}
    \end{align}
 From (\ref{2.21}) and (\ref{2.22}), we get
 \begin{align*}
     \begin{cases}
         -M^{-1}x^*= x-y\in (A^{-1}+\widetilde{B})(x^*/\lambda),\\
         -M^{-1} y^*= y-x\in (\widetilde{A}+B^{-1})(y^*/\lambda),
     \end{cases}
 \end{align*}
 i.e.,
 \begin{align}
     \begin{cases}
         x^*= \lambda J^{M^{-1}}_{\frac{A^{-1+\widetilde{B}}}{\lambda}}(0)= \lambda u^*,\\
         y^* =\lambda J^{M^{-1}}_{\frac{\widetilde{A}+B^{-1}}{\lambda}}(0)=\lambda v^*.
     \end{cases}
 \end{align}
    
\end{itemize}
    
\end{proof}
In the next results, we show the convergence of backward-backward splitting algorithm to problem (\ref{1.1}).
\begin{theorem}\label{thm5.3}
    Let $M:\mathcal{H}\to \mathcal{H}$ be an admissible preconditioner and $A,B:\mathcal{H}\to 2^\mathcal{H}$ be two maximal monotone operators with  $S\neq\emptyset$ and $\lambda >0$. Fix $x_0\in \mathcal{H}$ and for $n\ge 0$ set 
    \begin{align}\label{eq1.9}
 y_n=J_{\lambda B}^M x_n,~ x_{n+1}= J_{\lambda A}^M y_n.
    \end{align}
    Then we have the following:
    \begin{itemize}
        \item [(i)] The sequence $\{(x_n, y_n)\}$ converges weakly to a point in $S$.
        \item [(ii)] For every $(\bar{x},\bar{y})\in S$,
        \begin{align*}
        \sum_{n\in\mathbb{N}} \|(x_n-y_n)-(\bar{x}-\bar{y})\|_M^2<\infty,~\text{and}\\
        \sum_{n\in\mathbb{N}} \|(x_{n+1}-y_n)-(\bar{x}-\bar{y})\|_M^2<\infty.
        \end{align*}
        \item [(iii)] The sequence $\{(y_n-x_n, x_{n+1}-y_n)\}$ converges strongly to $\lambda(M^{-1}u^*, M^{-1}v^*).$
    \end{itemize}
\end{theorem}
\begin{proof}
Let $(x^*, y^*)\in S$. 
By Proposition \ref{pro2.1} and (\ref{eq1.9}), we have
\begin{align}
    \|x_{n+1}-x^*\|_M^2&= \left\|J_{\lambda A}^M J_{\lambda B}^M x_n -J_{\lambda A}^M J_{\lambda B}^M \bar{x}\right\|_M^2\nonumber\\
    & \le \left\| J_{\lambda B}^Mx_n- J_{\lambda B}^M\bar{x}\right\|_M^2-\left\|(I-J_{\lambda A}^M)J_{\lambda B}^M x_n- (I-J_{\lambda A}^M)J_{\lambda B}^M \bar{x}\right\|_M^2\nonumber\\
    &\le \left\|x_n-\bar{x}\right\|_M^2- \left\|(I-J_{\lambda B}^M)x_n- (I-J_{\lambda B}^M) \bar{x}\right\|_M^2\nonumber\\
    &~~-\left\|(I-J_{\lambda A}^M)J_{\lambda B}^M x_n- (I-J_{\lambda A}^M)J_{\lambda B}^M \bar{x}\right\|_M^2\nonumber\\
    &=\|x_n-\bar{x}\|_M^2-\|(x_n-y_n)-(\bar{x}-\bar{y})\|_M^2-\|(y_n-x_{n+1})-(\bar{y}-\bar{x})\|_M^2\nonumber,
\end{align}
which implies that
\begin{align}\label{3.28}
    \sum_{n\in \mathbb{N}}\|(x_n-y_n)-(\bar{x}-\bar{y})\|_M^2+\|(x_{n+1}-y_n)-(\bar{x}-\bar{y})\|_M^2\|\le \|x_0-\bar{x}\|^2_M,
\end{align}
which concludes (ii). 

From Proposition \ref{pro4.1}(i) and (v), we get
\begin{align}\label{3.29}
    \overline{y}= J_{\lambda B}^M \overline{x}= \overline{x}+\lambda M^{-1} u^*~\text{and}~\overline{x}=J_{\lambda A}^M \overline{y}= \overline{y}+\lambda M^{-1}v^*.
\end{align}
From (\ref{3.28}) and (\ref{3.29}), we obtain $y_n-x_n\to \overline{y}-\overline{x}=\lambda M^{-1}u^*$ and $x_{n+1}-y_n\to \overline{x}-\overline{y}=\lambda M^{-1}v^*$, which concludes (iii) and also the fact that $x_{n+1}-x_n\to 0$. 

Now, since $J_{\gamma A}^M J_{\gamma B}^M$ is nonexpansive, the sequence $\{x_n\}$ converges weakly to a fixed point $x$ of $J_{\gamma A}^M J_{\gamma B}^M$. Let $y=J_{\gamma B}^Mx$. From Proposition \ref{pro4.1} (i) and (ii), we obtain that $(x,y)\in S$. Hence, sequence $\{y_n\}$ converges weakly to $y$, as $y_n-x_n\to y-x$. Therefore, $\{(x_n,y_n)\}$ converges weakly to a point in $S$.  
\end{proof}
\begin{theorem}
Let $M:\mathcal{H}\to \mathcal{H}$ be an admissible preconditioner and $A,B:\mathcal{H}\to 2^\mathcal{H}$ be two maximal monotone operators such that $S=\emptyset$ and $\gamma >0$. Let the sequence $\{(x_n,y_n)\}$ be defined by (\ref{eq1.9}). Then $\|x_n\|\to \infty$ and $\|y_n\|\to \infty$. 
\end{theorem}
\begin{proof}
For every $n\in \mathbb{N}$, we have $x_n=T^n x_0$, where $T= J_{\lambda A}^M J_{\lambda B}^M$ is  and by \cite[Fact 2.2]{bauschke2005asymptotic}, $T$ is strongly nonexpansive. So $\|x_n\|\to \infty$ as $n\to \infty$. Similarly, we can show $\|y_n\|\to \infty$.  
\end{proof}
\section{An application}\label{sc5}
Let $\Gamma_0(\mathcal{H})$ be a collection of proper, convex and lower-semicontinuous from $\mathcal{H}$ to $(-\infty, \infty]$ and $f,g\in \Gamma_0(\mathcal{H})$. Consider the function
\begin{align}
    \Phi:\mathcal{H}\times\mathcal{H}\to (-\infty, \infty]: (x,y)\mapsto f(x)+g(y)+\frac{1}{2\lambda}\|x-y\|^2_M,
\end{align}
where $M:\mathcal{H}\to \mathcal{H}$ is a linear, bounded, self-adjoint and positive semi-definite operator. By taking $A$ and $B$ to be the convex subdifferential of $f$ and $g$, respectively, by Proposition \ref{pro4.1} and Theorem \ref{thm5.3}, the following results hold for the problem: $\min \Phi(\mathcal{H}\times\mathcal{H})$.
\begin{corollary}
    \begin{itemize}
        \item [(i)] $S=\operatorname{argmin}(\Phi)$.
        \item [(ii)] $E= \Fix(\prox_{\lambda g}^M \prox_{\lambda f}^M)$ and $F=\Fix(\prox_{\lambda f}^M \prox_{\lambda g}^M)$. 
        \item [(iii)] $S=(F\times E)\cap \mathcal{G}(\prox_{\lambda g}^M)$.
    \end{itemize}
\end{corollary}
\begin{corollary}
    Let $M:\mathcal{H}\to \mathcal{H}$ be an admissible preconditioner, $f,g\in \Gamma_0(\mathcal{H})$ with $S\neq\emptyset$ and $\lambda >0$. Fix $x_0\in \mathcal{H}$ and for $n\ge 0$ set
    \begin{align*}
        y_n=\prox_{\lambda g}^M x_n,~ x_{n+1}= \prox_{\lambda f}^M y_n.
    \end{align*}Then we have the following:
    \begin{itemize}
        \item [(i)] The sequence $\{(x_n, y_n)\}$ converges weakly to a point in $S$.
        \item [(ii)] For every $(\bar{x},\bar{y})\in S$,
        \begin{align*}
        \sum_{n\in\mathbb{N}} \|(x_n-y_n)-(\bar{x}-\bar{y})\|_M^2<\infty,~\text{and}\\
        \sum_{n\in\mathbb{N}} \|(x_{n+1}-y_n)-(\bar{x}-\bar{y})\|_M^2<\infty.
        \end{align*}
        \item [(iii)] The sequence $\{(y_n-x_n, x_{n+1}-y_n)\}$ converges strongly to $\lambda(M^{-1}u^*, M^{-1}v^*)$.
    \end{itemize}
\end{corollary}

\bibliographystyle{tfnlm}
\bibliography{interactnlmsample}
\end{document}